\documentclass[12pt,a4paper,reqno]{amsart}
\usepackage{amssymb}
\usepackage{amscd}
\usepackage{enumerate}
\usepackage{graphicx}
\usepackage{siunitx}
\usepackage{tikz-cd}
\usepackage{color}
\usetikzlibrary{arrows}
\numberwithin{equation}{section}

\usepackage{mathabx}

\usepackage{mathtools}%                  http://www.ctan.org/pkg/mathtools
\usepackage[tableposition=top]{caption}% http://www.ctan.org/pkg/caption
\usepackage{booktabs,dcolumn}%           http://www.ctan.org/pkg/dcolumn + http://www.ctan.org/pkg/booktabs

\DeclareFontFamily{OT1}{rsfs}{}
\DeclareFontShape{OT1}{rsfs}{n}{it}{<-> rsfs10}{}
\DeclareMathAlphabet{\mathscr}{OT1}{rsfs}{n}{it}

\theoremstyle{plain}

\newtheorem{theorem}{Theorem}[section]
\newtheorem{proposition}[theorem]{Proposition}
\newtheorem{lemma}[theorem]{Lemma}
\newtheorem{corollary}[theorem]{Corollary}

\theoremstyle{definition}

\newtheorem{remark}[theorem]{Remark}

\newcommand\R{\mathbb{R}}

\newcommand\eps{\varepsilon}

\usepackage{algorithm, algorithmic}

\begin{document}

\title[Blowup for high dimensional NLW]{Finite time blowup for high dimensional nonlinear wave systems with bounded smooth nonlinearity}

\author{Terence Tao}
\address{UCLA Department of Mathematics, Los Angeles, CA 90095-1555.}
\email{tao@math.ucla.edu}

%\email{}


\begin{abstract}  We consider the global regularity problem for nonlinear wave systems
$$ \Box u = f(u) $$
on Minkowski spacetime $\R^{1+d}$ with d'Alembertian $\Box := -\partial_t^2 + \sum_{i=1}^d \partial_{x_i}^2$, where the field $u \colon \R^{1+d} \to \R^m$ is vector-valued, and the nonlinearity $f \colon \R^m \to \R^m$ is a smooth function with $f(0)=0$ and all derivatives bounded; the higher-dimensional sine-Gordon equation $\Box u = \sin u$ is a model example of this class of nonlinear wave system.  For dimensions $d \leq 9$, it follows from the work of Heinz, Pecher, Brenner, and von Wahl that one has smooth solutions to this equation for any smooth choice of initial data.  Perhaps surprisingly, we show that this result is almost sharp, in the sense that for any $d \geq 11$, there exists an $m$ (in fact we can take $m=2$) and a nonlinearity $f \colon \R^m \to \R^m$ with all derivatives bounded, for which the above equation admits solutions that blow up in finite time.  The intermediate case $d=10$ remains open.
\end{abstract}

\maketitle

%%%%%%%%%%%%%%%%%%%%%%%%%%%%%%%%%%%%%%%%%%%%%%%%%

\section{Introduction}

This paper is concerned with nonlinear wave systems of the form
\begin{equation}\label{boxu}
 \Box u = f(u) 
\end{equation}
involving a vector-valued field $u\colon \R^{1+d} \to \R^m$ on Minkowski spacetime $\R^{1+d}$ with d'Alembertian $\Box := -\partial_t^2 + \sum_{i=1}^d \partial_{x_i}^2$, where the field $u\colon \R^{1+d} \to \R^m$ is vector-valued, and $f\colon \R^m \to \R^m$ is a smooth function with all derivatives bounded and $f(0)=0$.  A typical example of such a system is the higher-dimensional sine-Gordon equation
$$ \Box u = \sin u.$$
Suppose we are given an initial position $u_0\colon \R^{1+d} \to \R^m$ and an initial velocity $u_1\colon \R^{1+d} \to \R^m$, which are smooth and compactly supported.  Standard energy methods (see e.g. \cite[\S 1.4]{sogge-wave}) show that for any such data, there is a unique time $0 < T_* \leq \infty$ and a smooth solution $u: [0,T_*) \times \R^d \to \R^m$ to \eqref{boxu}, compactly supported in space for each time $t$, such that $u(0,x) = u_0(x)$ and $\partial_t u(0,x) = u_1(x)$; furthermore, if $T_* < \infty$, the solution $u$ cannot be smoothly extended to the time $t=T_*$ (in fact the norm $\|u(t)\|_{L^\infty_x(\R^d)}$ must go to infinity as $t \to T_*^-$).  We say that the equation \eqref{boxu} enjoys \emph{global regularity} if the maximal time of existence $T_*$ is infinite for any choice of smooth, compactly supported initial data $u_0, u_1$.  Actually, due to finite speed of propagation (see e.g. \cite[Proposition 3.3]{tao-book}), the requirement that $u_0$ and $u_1$ be compactly supported can be dropped without affecting the global regularity property.

The equation \eqref{boxu} is an extremely ``subcritical'' semilinear wave equation, since the nonlinearity $f$ exhibits no growth whatsoever at infinity.  As a consequence, there are plenty of \emph{a priori} bounds one can place on solutions to \eqref{boxu} with smooth compactly supported initial data.  For instance, suppose $T_*$ were finite.  From finite speed of propagation we see that $u(t)$ is supported in  a fixed compact region for all times $t \in [0,T_*)$.   In particular, $f(u)$ lies in the space $L^\infty_t L^2_x( [0,T_*) \times \R^d))$, which from \eqref{boxu} and the energy inequality shows that $u$ lies in the space $L^\infty_t H^1_x( [0,T_*) \times \R^d )$.  From the chain rule identity
$$ \nabla_x f(u) = (\nabla_{\R^m} f)(u) \nabla_x u,$$
where $(\nabla_{\R^m} f)(v)\colon \R^m \to \R^m$ is the derivative of $f\colon \R^m \to \R^m$ at a point $v \in \R^m$, we then conclude that $f(u)$ also lies in $L^\infty_t H^1( [0,T*) \times \R^d)$, and hence that $u$ lies in the space $L^\infty_t H^2_x([0,T_*) \times \R^d)$.  In the case of very low dimensions $d \leq 3$, Sobolev embedding then gives an $L^\infty_t L^\infty_x$ bound on $u$ which, when combined with energy methods, is sufficient to establish global regularity.

Naively one might expect to keep iterating the above procedure to handle arbitrarily large dimensions, but complications arise from the lower order terms in the iterated chain rule (or Faa di Bruno formula), which can ultimately be ``blamed'' on the phenomenon that $f(u)$ may oscillate at significantly higher frequencies than $u$.  For instance, the second derivative of $f(u)$ is given by the formula
$$ \nabla_x^2 f(u) = (\nabla_{\R^m} f)(u) \nabla_x^2 u + (\nabla_{\R^m}^2 f)(u) \nabla_x u \nabla_x u$$
where we do not specify exactly how to contract the various tensors displayed in this equation against each other for sake of exposition.  The existing $L^\infty_t H^2_x$ bound on $u$ allows us to place the lower order term $(\nabla_{\R^m}^2 f)(u) \nabla_x u \nabla_x u$ in $L^\infty_t L^2_x$ in four and fewer dimensions thanks to Sobolev embedding, but this is not immediately obvious in higher dimensions, so one cannot immediately upgrade the regularity of $u$ to $L^\infty_t H^3_x$.  Nevertheless, by increasingly sophisticated arguments \cite{heinz}, \cite{vonwahl}, \cite{pecher-2}, \cite{pecher-3}, \cite{pecher}, \cite{brenner}, \cite{wahl} in more complicated function spaces (such as Besov spaces), global regularity for \eqref{boxu} was established in dimensions $d \leq 9$.  Strictly speaking, the hypotheses on $m$ and $f$ in the above references were somewhat different in these references than those provided here (in particular, $f$ was allowed to exhibit some growth at infinity), but the arguments can be adapted to handle the setting under discussion; for the convenience of the reader we give such an argument in Appendix \ref{reg}.

The main result of this paper is to show that the condition $d \leq 9$ in these previous results is not merely technical, but is in fact nearly the correct threshold:

\begin{theorem}[Finite time blowup in high dimensions]\label{main}  Let $d \geq 11$ and $m \geq 2$ be integers.  Then there exists a smooth function $f\colon \R^m \to \R^m$ with all derivatives bounded, and a smooth solution $u: (0,\kappa] \times \R^d \to \R^m$ to \eqref{boxu} for some $\kappa > 0$ which cannot be smoothly continued to the spacetime origin $(0,0)$.
\end{theorem}

Applying time reversal symmetry and then shifting this solution in time by $\kappa$, and using finite speed of propagation to smoothly truncate the initial data to be compactly supported, we see that global regularity for \eqref{boxu} fails in eleven and higher dimensions for general nonlinearities $f$.   Somewhat frustratingly, neither the positive or negative results in this paper seem to extend to cover the intermediate case $d=10$, which remains open.  The argument in Theorem \ref{main} uses the strong Huygens principle in eleven dimensions, which is unavailable in ten dimensions, but this is most likely only a technical restriction.  More seriously, and as we shall see shortly, the numerology of exponents used in Theorem \ref{main} cannot satisfy all the constraints required for a consistent blowup ansatz in dimensions ten and lower.  

Our methods do not extend to $m=1$, basically because we cannot get any good injectivity properties of $u$ in this case, even after using spherical symmetry to perform a dimensional reduction.  In particular, Theorem \ref{main} does not directly establish finite time blowup for the model equation $\Box u= \sin u$ in eleven and higher dimensions.  However, they provide a \emph{barrier} to any attempt to prove global regularity for such an equation, by showing that such an attempt can only be successful if it genuinely uses some additional property of this model equation that is not shared by the more general systems \eqref{boxu}.

We now give an informal and non-rigorous description of the numerology underlying Theorem \ref{main}, with several of the notions in this description (such as the interpretation of the $\approx$ symbol) deliberately left vague.  For sake of this discussion let us restrict attention to the sine-Gordon equation $\Box u = \sin u$ in $d$ spatial dimensions.  The blowup ansatz we will use is as follows: for each frequency $N_j$ in a sequence $1 < N_1 < N_2 < N_3 < \dots$ of large quantities going to infinity, there will be a spacetime ``cube'' $Q_j = \{ (t,x): t \sim \frac{1}{N_j}; x = O(\frac{1}{N_j})\}$ on which the solution $u$ oscillates with ``amplitude'' $N_j^\alpha$ and ``frequency'' $N_j$, where $\alpha>0$ is an exponent to be chosen later; this ansatz is of course compatible with the uncertainty principle.  Since $N_j^\alpha \to \infty$ as $j \to \infty$, this will create a singularity at the spacetime origin $(0,0)$.  To make this ansatz plausible, we wish to make the oscillation of $u$ on $Q_j$ driven primarily by the forcing term $\sin u$ at $Q_{j-1}$.  Thus, by Duhamel's formula, we expect a relation roughly of the form
$$ u(t,x) \approx \int \frac{\sin((s-t)\sqrt{-\Delta})}{\sqrt{-\Delta}} \sin(1_{Q_{j-1}} u(s)) (x)\ ds$$
on $Q_j$, where $\frac{\sin((s-t)\sqrt{-\Delta})}{\sqrt{-\Delta}}$ is the usual free wave propagator, and $1_{Q_{j-1}}$ is the indicator function of $Q_{j-1}$.

On $Q_{j-1}$, $u$ oscillates with amplitude $N_{j-1}^\alpha$ and frequency $N_{j-1}$, we expect the derivative $\nabla_{t,x} u$ to be of size about $N_{j-1}^{\alpha+1}$, and so from the principle of stationary phase we expect $\sin(u)$ to oscillate at frequency about $N_{j-1}^{\alpha+1}$.  Since the wave propagator $\frac{\sin((s-t)\sqrt{-\Delta})}{\sqrt{-\Delta}}$ preserves frequencies, and $u$ is supposed to be of frequency $N_j$ on $Q_j$ we are thus led to the requirement
\begin{equation}\label{nj-form}
N_j \approx N_{j-1}^{\alpha+1}.
\end{equation}
Next, when restricted to frequencies of order $N_{j}$, the propagator $\frac{\sin((s-t)\sqrt{-\Delta})}{\sqrt{-\Delta}}$ ``behaves like'' $N_{j}^{\frac{d-3}{2}} (s-t)^{\frac{d-1}{2}} A_{s-t}$, where $A_{s-t}$ is the spherical averaging operator
$$ A_{s-t} f(x) := \frac{1}{\omega_{d-1}} \int_{S^{d-1}} f(x + (s-t)\theta)\ d\theta$$
where $d\theta$ is surface measure on the unit sphere $S^{d-1}$, and $\omega_{d-1}$ is the volume of that sphere; see e.g. \cite[\S 1.1]{sogge-wave}. In our setting, $s-t$ is comparable to $1/N_{j-1}$, and so we have the informal approximation
$$ u(t,x) \approx N_j^{\frac{d-3}{2}} N_{j-1}^{-\frac{d-1}{2}} \int_{s \sim 1/N_{j-1}} A_{s-t} \sin(u(s))(x)\ ds$$
on $Q_j$.

Since $\sin(u(s))$ is bounded, $A_{s-t} \sin(u(s))$ is bounded as well.  This gives a (non-rigorous) upper bound
$$ u(t,x) \lessapprox N_j^{\frac{d-3}{2}} N_{j-1}^{-\frac{d-1}{2}} \frac{1}{N_{j-1}} $$
which when combined with our ansatz that $u$ has amplitude about $N_j^\alpha$ on $Q_j$, gives the constraint
$$ N_j^\alpha \lessapprox N_j^{\frac{d-3}{2}} N_{j-1}^{-\frac{d-1}{2}} \frac{1}{N_{j-1}} $$
which on applying \eqref{nj-form} gives the further constraint
$$ \alpha(\alpha+1) \leq \frac{d-3}{2} (\alpha+1) - \frac{d-1}{2} - 1 $$
which can be rearranged as
$$ \left(\alpha - \frac{d-5}{4}\right)^2 \leq \frac{d^2-10d-7}{16}.$$
It is now clear that the optimal choice of $\alpha$ is 
\begin{equation}\label{alpha-dig}
\alpha = \frac{d-5}{4},
\end{equation}
and this blowup ansatz is only self-consistent when 
\begin{equation}\label{d10}
  \frac{d^2-10d-7}{16} \geq 0
\end{equation}
or equivalently if $d \geq 11$.

\begin{figure} [t]
\centering
\includegraphics{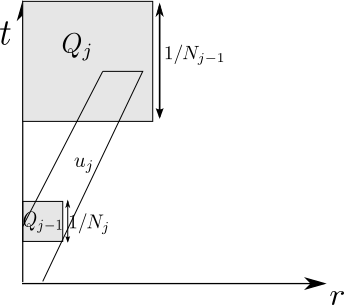}
\caption{A schematic depiction of the support of the component $u_j$ of $u$ using polar coordinates $(t,r) = (t,|x|)$; it evolves like a free wave except for a forcing term in $Q_{j-1}$, and is concentrated on $Q_j$, where it has frequency about $N_j$ and amplitude about $N_j^\alpha$.  Our construction will be in an odd spatial dimension, so that the strong Huygens principle is available, restricting the support of $u_j$ to the neighbourhood of a light cone.}
\label{fig:freq}
\end{figure}

To turn this ansatz into an actual blowup example, we will construct $u$ as the sum of various functions $u_j$ that solve the wave equation with forcing term in $Q_{j+1}$, and which concentrate in $Q_j$ with the amplitude and frequency indicated by the above heuristic analysis; see Figure \ref{fig:freq}.  The remaining task is to show that $\Box u$ can be written in the form $f(u)$ for some $f$ with all derivatives bounded.  For this one needs some injectivity properties of $u$ (after imposing spherical symmetry to impose a dimensional reduction on the domain of $u$ from $d+1$ dimensions to $1+1$).  This requires one to construct some solutions to the free wave equation that have some unusual restrictions on the range (for instance, we will need a solution taking values in the plane $\R^2$ that avoid one quadrant of that plane).  Such solutions will be constructed in Section \ref{freesec}.

\begin{remark}  Our nonlinearity $f$ is non-Hamiltonian (or non-Lagrangian) in the sense that it is not of the form $f = \nabla F$ for some smooth potential function $F\colon \R^m \to \R$.  The requirement that $f$ be Hamiltonian would impose an additional constraint
$$ \partial_\alpha \langle \partial_\beta u, \Box u \rangle_{\R^m} =  \partial_\beta \langle \partial_\alpha u, \Box u \rangle_{\R^m} $$
on the solution $u$ (which arises from the Clairaut identity $\partial_\alpha \partial_\beta F(u) = \partial_\beta \partial_\alpha F(u)$).  We believe though that for sufficiently large $m$, one can adapt the construction here (in combination with the Nash embedding theorem, as in \cite{tao-super}) to ensure that the solution $u$ obeys this constraint, and to thus obtain a Hamiltonian counterexample for Theorem \ref{main} in eleven and higher dimensions.  We will not pursue this matter here.
\end{remark}

\begin{remark} The positive results established in Appendix \ref{reg} rely on Strichartz estimates, while the negative results involve solutions for which the Strichartz estimates are not sharp; this is related to the fact that the solution component $u_j$ that we construct in our proof of Theorem \ref{main} does not occupy all of $Q_{j-1}$, but is instead concentrated near a light cone (see Figure \ref{fig:freq}), thus ``wasting a dimension'' in some sense.  This may potentially explain the gap between the positive and negative results.  In any event, closing the gap in the positive direction seems to require estimates that go beyond the Strichartz estimates, whereas closing the gap in the negative direction may require constructions of solutions in which the Strichartz estimates are closer to being sharp.
\end{remark}

\begin{remark}  It is possible that one could extend the positive results in Appendix \ref{reg} to slightly higher dimensions, such as $d=10$ or $d=11$, if one imposes some additional decay on higher derivatives of $f$ at infinity, e.g. if one requires that $|\nabla_{\R^m}^j f(x)| \lesssim_j (1+|x|)^{-j}$ for all $x \in \R^m$, or even requiring that $f$ be compactly supported.  This would go a fair way towards eliminating, or at least attenuating, the blowup ansatz used in Theorem \ref{main}.  However, the fundamental issue remains in this setting that the nonlinearity $f(u)$ can still oscillate at a significantly higher frequency than $u$ itself (although now this oscillation will be largely confined to a small neighbourhood of the zero set $\{u=0\}$ of $u$), and it does not appear likely that the positive results can be extended to arbitrarily high dimension even with such strong hypotheses on the nonlinearity.
\end{remark}
 
The author is supported by NSF grant DMS-1266164 and by a Simons Investigator Award.  The author is also indebted to Jeffrey Rauch for suggesting this question, and Michael Peake and Claude Zuily for some corrections and suggestions.  The author is particularly indebted to Claude Zuily for pointing out a numerical error in the original version of Appendix \ref{reg}.

\subsection{Notation}

We use $|x|$ to denote the Euclidean norm of a vector $x$.  If $f\colon \R \to \R^m$ is a smooth function, we use $f^{(j)}\colon \R \to \R^m$ to denote the $j^{\operatorname{th}}$ derivative of $f$.

We use $X \lesssim Y$, $Y \gtrsim X$, or $X = O(Y)$ to denote the estimate $|X| \leq CY$ for an absolute constant $C$, and $X \sim Y$ to denote the estimates $X \lesssim Y \lesssim X$.  We will often require the implied constant $C$ in the above notation to depend on additional parameters, which we will indicate by subscripts (unless explicitly omitted), thus for instance $X \lesssim_j Y$ denotes an estimate of the form $|X| \leq C_j Y$ for some $C_j$ depending on $j$.

\section{A lemma on spherically symmetric functions}

We will be working with smooth spherically symmetric functions $u: I \times \R^d \to \R^m$ on various intervals $I$, that is to say functions $u(t,x)$ that depend only on the time $t$ and on the magnitude $r=|x|$ of the spatial variable.  As is well known, one can perform a dimensional reduction, using the coordinates $(t,r) = (t,|x|)$ instead of $(t,x)$, to view such functions as functions on the strip $I \times [0,+\infty)$ rather than $I \times \R^d$.  But when one does so, one creates a degeneracy at the time axis $r=0$; more precisely, smooth spherically symmetric functions $u$ must necessarily have a vanishing gradient $\nabla_x u(t,0)$ at the spatial origin, which makes $\partial_r u(t,r)$ vanish at $r=0$.  This vanishing of the first derivative is undesirable for our applications, as we will need to invert the map $u$ near the time axis using the inverse function theorem.  Because of this, it will be more convenient to work with the variable $y := r^2 = |x|^2$ rather than $r$.  In this section we give some simple calculus lemmas that manage this change of variables.  We begin with the one-dimensional scalar case $d=m=1$, in which case spherical symmetry just means that the functions involved are even functions of the spatial variable $x$.

\begin{lemma}[One-dimensional spherically symmetric functions]\label{rad}  Let $f\colon \R \to \R$ be a smooth even function obeying the bounds 
\begin{equation}\label{cantry}
|f^{(j)}(x)| \lesssim_j 1
\end{equation}
for all $j \geq 0$ and $x \in \R$.  Let $F: [0,+\infty) \to \R$ be the function defined by setting $F(x^2) := f(x)$ for all $x \in \R$.  Then $F$ is smooth and
\begin{equation}\label{fajy}
 |F^{(j)}(y)| \lesssim_j (1+y^{1/2})^{-j}
\end{equation}
for all $j \geq 0$ and $y \in [0,+\infty)$.
\end{lemma}

It is possible to reverse the implication and deduce \eqref{cantry} from \eqref{fajy}, but we will not need to do so here.  The factor $y^{1/2}$ in \eqref{fajy} naturally arises from the chain rule, since $\frac{dy}{dx} = 2 y^{1/2}$ when $y = x^2$; the hypothesis that $f$ is even gives an improvement in the range $|y| \leq 1$ by replacing $y^{1/2}$ with $1$.

\begin{proof}  Clearly $F$ is smooth away from $0$.  By the fundamental theorem of calculus it suffices to prove the uniform bounds \eqref{fajy} for $y>0$.  We do this by induction on $j$.  The $j=0$ case is trivial, so suppose that $j \geq 1$ and the claim has already been proven for $j-1$.  Differentiating the identity $F(x^2) = f(x)$, we have
\begin{equation}\label{fp}
 2x F'(x^2) = f'(x) 
\end{equation}
for any $x$.  As $f$ is even, $f'$ vanishes at zero, hence $f'(x) = x\int_0^1 f''(tx)\ dt$ by the fundamental theorem of calculus.  We conclude that
$$ F'(x^2) = \frac{1}{2} \int_0^1 f''(tx)\ dt$$
for $x \neq 0$; the claim also follows for $x=0$ by Taylor expansion of $f$ at the origin.

The function $\tilde f \colon x \mapsto \frac{1}{2} \int_0^1 f''(tx)\ dt$ is smooth and even, and by \eqref{cantry} and the triangle inequality, all the derivatives of $\tilde f$ are bounded.  Applying the induction hypothesis with $j$ replaced by $j-1$ and $f$ replaced by $\tilde f$, we obtain the bound
$$ |F^{(j)}(y)| \lesssim_j (1+y^{1/2})^{-(j-1)}$$
which gives the claim when $y \leq 1$.

It remains to establish \eqref{fajy} in the region $2^k \leq y \leq 2^{k+1}$ for any $k \geq 0$.  From \eqref{fp} we have
$$ F'(x^2) = 2^{-k/2-1} \frac{\eta( \frac{x}{2^{k/2}} )}{x/2^{k/2}} f'(x) $$
for $2^k \leq x^2 \leq 2^{k+1}$ and some smooth even cutoff function $\eta\colon \R \to \R$ that equals one on $[-2,-1] \cup [1,2]$ and vanishes outside of $[-4,-1/2] \cup [1/2,1/4]$.  The function $f_k \colon x \mapsto \frac{\eta( \frac{x}{2^{k/2}} )}{x/2^{k/2}} f'(x)$ is smooth and even, and by \eqref{cantry} and the product rule has all derivatives bounded.  Applying the induction hypothesis with $j$ replaced by $j-1$ and $f$ replaced by $f_k$, we obtain the bound
$$ |F^{(j)}(y)| \lesssim_j 2^{-k/2-1} (1+y^{1/2})^{-(j-1)}$$
for $2^k \leq y \leq 2^{k+1}$, which gives \eqref{fajy} as required.
\end{proof}

Now we may easily generalise to higher dimensions:

\begin{corollary}\label{cod}  Let $I \subset \R$ be an interval, let $d,m \geq 1$, and let $u: I \times \R^d \to \R^m$ be a smooth spherically symmetric function obeying the bounds
\begin{equation}\label{cantry-2}
|\partial^k_t \nabla_{x}^j u(t,x)| \lesssim_{j,k} A B^{-j} C^{-k}
\end{equation}
for all $j,k\geq 0$ and $(t,x) \in I \times \R^d$, and some $A,B,C>0$.  Let $U: I \times [0,+\infty) \to \R^m$ be the function defined by setting $U(t,|x|^2) := u(t,x)$ for all $(t,x) \in I \times \R^d$.  Then $U$ is smooth and
\begin{equation}\label{fajy-2}
|\partial_k^t \partial_y^j U(t,y)| \lesssim_{j,k,d,m} A (1 + B y^{1/2})^{-j} C^{-k}
\end{equation}
for all $j \geq 0$ and $(t,y) \in I \times [0,+\infty)$.
\end{corollary}

Again, one can establish a converse to this claim, but we will not need to do so here.

\begin{proof}  We can rescale $A=B=C=1$.  By breaking into components we may assume $m=1$.  By restricting to the plane $\{ (t, x_1 e_1): t \in I, x_1 \in \R\}$ we may assume that $d=1$.  The claim then follows by applying Lemma \ref{rad} to $u$ and its time derivatives.
\end{proof}

\section{A solution to the free wave equation}\label{freesec}

To prove Theorem \ref{main}, it is clear that we can restrict to the case $m=2$, since the general case $m \geq 2$ can then be established by embedding $\R^2$ in $\R^m$ (and trivially extending the nonlinearity $f$ from $\R^2$ to $\R^m$).  Similarly, we can restrict to the case $d=11$, as the $d>11$ case then follows by adding dummy spatial variables.

To build the solution $u$ in the case $d=11, m=2$, we will need a certain ``non-degenerate'' solution $v\colon \R^{1+11} \to \R^2$ to the free wave equation $\Box v = 0$, which is easy to construct when there are at least two degrees of freedom and the number of spatial dimensions is odd.  We begin by constructing a scalar solution with a certain unusual positivity property.

\begin{proposition}[Scalar free wave]\label{vorp-1}  Let $d$ be an odd natural number.  Then there exists a smooth spherically symmetric solution $v_1\colon \R^{1+d} \to \R$ to the free wave equation
$$ \Box v_1 = 0 $$
which is compactly supported in space for each time, and obeys the following additional properties:
\begin{itemize}
\item[(i)]  There exists $\eps > 0$ such that $v_1(t,x)$ is strictly positive whenever $t \geq 0$ and $t-\eps \leq |x| \leq t + \eps$.  (In particular, $v_1(0,0)$ is strictly positive.)
\item[(ii)]  The derivatives $\partial_t v_1(0,0)$, $\partial_{tt} v_1(0,0)$, and\footnote{Strictly speaking, one should write $\partial_y V_1(0,0)$ here instead of $\partial_y v_1(0,0)$, where $V_1: \R \times [0,+\infty) \to \R$ is the function such that $V_1(t,|x|^2) := v_1(t,x)$ for all $(t,x) \in \R^{1+d}$.} $\partial_{y} v_1(0,0)$ are negative.
\end{itemize}
\end{proposition}

We remark in connection with the requirement (i) that in dimensions $d=1,3$ it is easy to make $v_1$ non-negative everywhere due to the positivity of the fundamental solution in this setting.  However, as is well known, in higher (odd) dimensions the fundamental solution contains derivatives, and it turns out not to be possible to ensure that $v_1$ is non-negative everywhere while still being compactly supported in space.  Fortunately, for our application to Proposition \ref{vorp} below, we can ``retreat'' to a neighbourhood of the light cone $|x|=t$ for the purposes of retaining positivity.  It may be possible to extend this proposition to even dimensions $d$, perhaps by combining the proof techniques below with the method of descent, but we will not pursue this matter here (since we will ultimately only need the $d=11$ case in any event).

\begin{proof}
Writing (by abuse of notation) $v_1(t,x) = v_1(t,r)$ with $r = |x|$, we see that $v_1\colon \R \times \R \to \R$ needs to be smooth, even in $r$ and obey the equation
$$ -\partial_{tt} v_1 + \partial_{rr} v_1 + \frac{d-1}{r} \partial_r v_1 = 0.$$
By repeatedly using the ``ladder operator'' identity
$$ \left( -\partial_{tt} + \partial_{rr} + \frac{d-1}{r} \partial_r \right) \frac{1}{r} \partial_r =  
\frac{1}{r} \partial_r \left( -\partial_{tt} + \partial_{rr} + \frac{d-3}{r} \partial_r \right) $$
for any $d$, we see that
$$ \left( -\partial_{tt} + \partial_{rr} + \frac{d-1}{r} \partial_r \right) \left(\frac{1}{r} \partial_r\right)^k =  
\left(\frac{1}{r} \partial_r\right)^{k} \left( -\partial_{tt} + \partial_{rr} \right)$$
where $k$ is the natural number $k := \frac{d-1}{2}$.
In particular, if we set
\begin{equation}\label{v1tr}
 v_1(t,r) := \left(\frac{1}{r} \partial_r\right)^k (g(t+r) + g(t-r))
\end{equation}
for some smooth $g\colon \R \to \R$ to be chosen later, then $v_1$ will be a smooth function of $t$ and $r$ that is even in $r$ (there is no singularity at $r=0$, since the operator $\frac{1}{r} \partial_r$ preserves the space of smooth even functions), and thus also a smooth spherically symmetric function of $t$ and $x$ (again, there is no singularity at $x=0$, as can be seen by applying Taylor's theorem with remainder\footnote{More precisely, Taylor expansion and the requirement of being even in $r$ gives an expansion of the form $v_1(t,r) = \sum_{i=1}^{k} c_i(t) r^{2i} + r^{2k} F(t,r)$ for some smooth coefficients $c_i(t)$ and smooth remainder $F(t,r)$, which ensures that $v_1(t,x)$ is $2k$ times continuously differentiable for any $k$.}); also, if $g$ is compactly supported, then $v_1$ will be compactly supported in space for each time $t$.

From the product rule and a routine induction on $k$ we see that
$$
\left(\frac{1}{r} \partial_r\right)^k = \sum_{i=0}^{k-1}  \frac{(k+i-1)!}{2^i(k-i-1)!i!} \frac{(-1)^i}{r^{k+i}} \partial_r^{k-i}.$$
In particular, we have
\begin{equation}\label{vater}
v_1(t,r) = \sum_{i=0}^{k-1} \frac{(k+i-1)!}{2^i(k-i-1)!i!} \frac{(-1)^i g^{(k-i)}(t+r) + (-1)^k g^{(k-i)}(t-r)}{r^{k+i}}.
\end{equation}
Specialising to the diagonal $(t,r) = (x/2,x/2)$ for $x>0$, we obtain
$$ v_1(x/2,x/2)
= 2^k \sum_{i=0}^{k-1} \frac{(k+i-1)!}{(k-i-1)!i!} \frac{(-1)^i g^{(k-i)}(x) + (-1)^k g^{(k-i)}(0)}{x^{k+i}}.$$
But from the product rule one has
$$ \sum_{i=0}^{k-1} \frac{(k+i-1)!}{(k-i-1)!i!} \frac{(-1)^i g^{(k-i)}(x)}{x^{k+i}}  
= \partial_x^{k-1} \left( \frac{g'(x)}{x^k} \right)$$
and
\begin{align*}
 \sum_{i=0}^{k-1} \frac{(k+i-1)!}{(k-i-1)!i!} \frac{(-1)^k g^{(k-i)}(0)}{x^{k+i}}  
&= - \partial_x^{k-1} \left( \frac{\sum_{i=0}^{k-1} \frac{1}{(k-i-1)!} g^{(k-i)}(0) x^{k-i-1}}{x^k} \right) \\
&= - \partial_x^{k-1} \left( \frac{\sum_{i=0}^{k-1} \frac{1}{i!} g^{(i+1)}(0) x^i}{x^k} \right) 
\end{align*}
and thus
$$ v_1(x/2,x/2) = 2^k \partial_x^{k-1} \left( \frac{g'(x) - \sum_{i=0}^{k-1} \frac{1}{i!} g^{(i+1)}(0) x^i}{x^k} \right)$$
for $x > 0$, and hence also for $x = 0$ by continuity (after removing the singularity at $x=0$ for the fraction on the right-hand side).
From this formula, we see that if we wish to have $v_1(x/2,x/2) > 0$ for all $x \geq 0$, it suffices to select $g'$ to have the form
\begin{equation}\label{rip}
 g'(x) = P(x) + x^k R(x) 
\end{equation}
for $x \geq 0$ for some polynomial $P$ of degree at most $k-1$ (which must then necessarily be given by the Taylor expansion $P(x) = \sum_{i=0}^{k-1} \frac{1}{i!} g^{(i+1)}(0) x^i$) and some smooth function $R(x)$ with $R^{(k-1)}(x) > 0$ for all $x \geq 0$.  

We set $P(x) := (-1)^k x^{k-1}$ and $R: [0,+\infty) \to \R$ to be any smooth function with $R(x) = (-1)^{k-1} \frac{1}{x}$ for all $x \geq 1$, and $R^{(k-1)}(x) > 0$ for all $x \geq 0$; such a function is easily constructed by choosing $R^{(k-1)}: [0,+\infty) \to \R$ to be any positive smooth function with $R^{(k-1)}(x) = \frac{(k-1)!}{x^k}$ for $x \geq 1$, and then integrating $R^{(k-1)}$ $k-1$ times from infinity to obtain $R$.  This defines $g'$ on $[0,+\infty]$, which vanishes on $[1,+\infty)$; we integrate this and then extend in a suitable fashion to the real line to obtain a smooth compactly supported $g\colon \R \to \R$ whose associated function $v_1$ given by \eqref{v1tr} is positive on the diagonal $\{ (x/2,x/2): x \geq 0 \}$.  Also, from construction we have $g^{(k)}(0) = P^{(k-1)}(0) = (-1)^k (k-1)!$; inserting this into \eqref{vater} we see that
$$ v_1(t,r) = \frac{(k-1)! + O(\eps)}{t^k} + O\left( \frac{1}{t^{k+1}} \right)$$
for any $\eps > 0$, any $t$ that is sufficiently large depending on $\eps$, and any $r$ between $t-\eps$ and $t+\eps$, where the implied constants can depend on $k, R, g$.  In particular, for $\eps>0$ small enough, this gives the conclusion (i) for sufficiently large $t$, and the conclusion for smaller $t$ follows from the positivity of $v_1$ on the diagonal and continuity (shrinking $\eps$ as necessary).

Finally, if we differentiate \eqref{v1tr} in time we obtain
$$ \partial_t v_1(t,r) := \left(\frac{1}{r} \partial_r\right)^k (g'(t+r) + g'(t-r)),$$
and then on performing a Taylor expansion of $g'$ and \eqref{rip} we see that
$$ \partial_t v_1(0,0) = \frac{2^{k+1} k!}{(2k)!} g^{(2k+1)}(0) = 2^{k+1} R^{(k)}(0).$$
Differentiating twice in time instead of once similarly gives
$$ \partial_{tt} v_1(0,0) = \frac{2^{k+1} k!}{(2k)!} g^{(2k+2)}(0) = 2^{k+1} \frac{2k+1}{k+1} R^{(k+1)}(0).$$
Finally, if we do not differentiate in time at all in \eqref{v1tr}, one final Taylor expansion gives
$$ \partial_{rr} v_1(0,0) =  \frac{2^{k+2} (k+1)!}{(2k+2)!} g^{(2k+2)}(0) = 2^{k+1} \frac{1}{k+1} R^{(k+1)}(0).$$
Since $R^{(k-1)}$ was chosen arbitrarily near zero (subject to being smooth and positive), we can easily ensure that $R^{(k)}(0)$ and $R^{(k+1)}(0)$ are negative, giving (ii) (since $\partial_y v_1(0,0) = \frac{1}{2} \partial_{rr} v_1(0,0)$).
\end{proof}

Now we add a second component $v_2$ to $v_1$ to obtain a vector-valued solution $v\colon \R^{1+d} \to \R^2$ to the wave equation obeying some technical conditions, mostly relating to the range of $v$ (see Figure \ref{fig:v}).

\begin{proposition}[Vector-valued free wave]\label{vorp}  Let $d$ be an odd natural number.  Then
there exists a smooth spherically symmetric solution $v\colon \R^{1+d} \to \R^2$ to the free wave equation
$$ \Box v = 0 $$
which is compactly supported in space for each time, and obeys the following additional properties:
\begin{itemize}
\item[(iii)] 
If we write $v = (v_1,v_2)$, then $v_1(0,0)$ and $v_2(0,0)$ are both positive.  Furthermore, for any $(t,x) \in\R^{1+d}$, we have either $v_1(t,x) \geq 0$ or $v_2(t,x) \geq 0$ (thus $v$ avoids the lower left quadrant of $\R^2$).
\item[(iv)] If we let $V\colon \R \times [0,+\infty) \to \R^2$ be the function defined by $V(t,|x|^2) := v(t,x)$ for $(t,x) \in \R^{1+d}$, then $\partial_t V(0,0)$ is a negative multiple of $(1,0)$, while both components of both $\partial_{tt} V(0,0)$ and $\partial_{y} V(0,0)$ are negative (where $y$ denotes the second variable of $V$).  
\item[(v)]  There exist constants $C,c>0$ such that whenever $(t,y) \in \R \times [0,+\infty)$ is such that $|V(t,y)-V(0,0)| \leq c$, we have the bounds
\begin{equation}\label{vat}
 V_1(0,0) - C( |t| + |y| ) \leq V_1(t,y) \leq V_1(0,0) - c(|t| + y) 
\end{equation}
and
\begin{equation}\label{vat-2} V_2(0,0) - C( |t|^2 + y ) \leq V_2(t,y) \leq V_2(0,0) - c(|t|^2 + y) 
\end{equation}
for the components $V_1(t,y), V_2(t,y)$ of $V(t,y)$. (In particular, this implies that $V(t,y) \neq V(0,0)$ whenever $(t,y) \neq (0,0)$.)
\item[(vi)] $V$ is supported in the region $\{ (t,y): y = (t+O(1))^2\}$, and one has the dispersive bounds
\begin{equation}\label{vos}
 |\nabla_t^j \nabla_y^k V(t,y)| \lesssim_{d,j,k} (1+|t|)^{-\frac{d-1}{2}-k}
\end{equation}
for all $(t,y) \in \R \times [0,+\infty)$ and $j,k \geq 0$, where we allow the implied constants to depend on $d$ and $v$.
\end{itemize}
\end{proposition}

\begin{proof}  
We set the first component $v_1$ of $v$ to be the function from Proposition \ref{vorp-1}.  For the second component $v_2$, we first choose $w \colon \R^{d} \to \R$ to be a smooth spherically symmetric function supported on $\{ x \in \R^{d}: |x| \leq \eps \}$ (where $\eps$ is the quantity from Proposition \ref{vorp-1}(i)), with the property that the Fourier transform
$$\hat w (\xi) := \int_{\R^{d}} w(x) e^{-i x \cdot \xi}\ dx = \int_{\R^{d}} w(x) \cos(x \cdot \xi)\ dx$$ 
is non-negative, and is not identically zero; in particular $w(0) = \frac{1}{(2\pi)^d} \int_{\R^{d}} \hat w(\xi)\ d\xi$ is strictly positive.  Such a $w$ can be constructed by taking an arbitrary non-zero smooth spherically symmetric real function and convolving it with itself.  We then let $v_2 \colon \R^{1+d} \to \R$ be the solution to the free wave equation with initial position $w$ and initial velocity $0$; more explicitly, we have
$$ v_2(t,x) := \frac{1}{(2\pi)^{d}} \int_{\R^{d}} \hat w(\xi) \cos( x \cdot \xi ) \cos(t|\xi|)\ d\xi.$$
Next, from Taylor expansion we see that
$$ \partial_t v_2(0,0) = 0$$
and
$$ \partial_{rr} v_2(0,0) = - \frac{1}{(2\pi)^{d}} \int_{\R^{d}} \hat w(\xi) \xi_1^2\ d\xi < 0$$
and
$$ \partial_{tt} v_2(0,0) = - \frac{1}{(2\pi)^{d}} \int_{\R^{d}} \hat w(\xi) |\xi|^2\ d\xi < 0$$
which, when combined with Proposition \ref{vorp-1}(ii), implies that $\partial_t V(0,0)$ is a negative multiple of $(1,0)$, and $\partial_y V(0,0)$ and $\partial_{tt} V(0,0)$ have both components negative, giving (iv).

From the triangle inequality we see that
$$ |v_2(t,x)| \leq \frac{1}{(2\pi)^{d}} \int_{\R^{d}} \hat w(\xi)\ d\xi = v_2(0,0) = w(0) > 0$$
for all $(t,x)$, with strict inequality for $(t,x) \neq (0,0)$ (because $\hat w$ cannot be supported on the measure zero set $\{ \xi: \cos(x \cdot \xi) \cos(t |\xi|) = \pm 1 \}$.  Since $v_2(t,x)$ decays to zero as $(t,x) \to \infty$, we conclude from a compactness argument that for any $\delta>0$, there exists $c>0$ such that $|(t,x)| \leq \delta$ whenever $|v_2(t,x)-v_2(0,0)| \leq c$.  In particular, after adjusting $\delta$ and $c$ appropriately, we have $|(t,y)| \leq \delta$ whenever $|V(t,y) - V(0,0)| \leq c$.  On the other hand, for $|(t,y)| \leq \delta$, we then see from two-dimensional Taylor expansion with remainder that
$$ V_1(t,y) = V_1(0,0) + (\partial_t V_1)(0,0) t + (\partial_y V_1)(0,0) y + O( \delta (|t|+y) )$$
and
$$ V_2(t,y) = V_2(0,0) + (\partial_t V_2)(0,0) t + \frac{1}{2} (\partial_{tt} V_2)(0,0) + (\partial_y V_2)(0,0) y + O( \delta (|t|^2+y) )$$
and the claims \eqref{vat}, \eqref{vat-2} then follow (for $\delta$ small enough) from (iv).

We have $v_2(0,0) = w(0) > 0$, and from Proposition \ref{vorp-1}(i) we have $v_1(0,0) > 0$, giving the first part of (iii).  Now we turn to the second part.  From the strong Huygens principle (see e.g. \cite[\S 1.1]{sogge-wave}); compare also with the explicit formula \eqref{v1tr}) we see that $v_2$ is non-vanishing only when $t-\eps \leq r \leq t+\eps$, but from Proposition \ref{vorp-1}(i) we know that $v_1$ is positive in this region.  This gives the second part of (iii).

Finally, we show (vi).  Using standard dispersive estimates for the free wave equation (see e.g. \cite[Theorem 1.1]{sogge-wave}), we see that
\begin{equation}\label{vodo}
 |v(t,x)| \lesssim_d (1+|t|)^{-\frac{d-1}{2}}
\end{equation}
for all $(t,x) \in \R^{1+d}$, and more generally
\begin{equation}\label{solar}
|\nabla_t^j \nabla_x^k v(t,x)| \lesssim_{d,j,k} (1+|t|)^{-\frac{d-1}{2}}
\end{equation}
for all $(t,x) \in \R^{1+d}$ and $j,k \geq 0$.  By Corollary \ref{cod}, this implies that
$$
 |\nabla_t^j \nabla_y^k V(t,y)| \lesssim_{d,j,k} (1+|t|)^{-\frac{d-1}{2}} (1+y^{1/2})^{-k}.
$$
On the other hand, from the strong Huygens principle\footnote{Even without the strong Huygens principle, one could still obtain the claimed estimates using the commuting vector fields method, see e.g. \cite{sogge-wave}.} we know that the left-hand side vanishes unless $y^{1/2} = t + O(1)$.  The claim follows.
\end{proof}

\begin{figure} [t]
\centering
\includegraphics{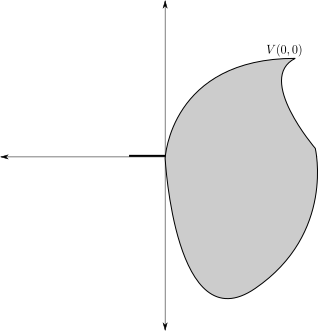}
\caption{A schematic depiction of the image $v([0,+\infty) \times \R^d) = V([0,+\infty)^2)$ of $v$ (or of $V$) for non-negative times; this image consists of the shaded region plus an interval protruding to the left of the origin (the latter arising because $v_2$ is only supported in the region where $v_1$ is strictly positive).   The two key features to note here are the ``corner'' of the image at $v(0,0)=V(0,0)$, with the image being locally contained purely in the lower left quadrant of $V(0,0)$; and the more global feature that the image completely avoids the lower left quadrant of the origin $(0,0)$.}
\label{fig:v}
\end{figure}

\section{Constructing the solution}

We are now ready to construct the solution $u$ for Theorem \ref{main}.  As discussed in the previous section, we may assume that $d=11$ and $m=2$.  Let $v\colon \R^{1+11} \to \R^2$ be the solution to the free wave equation constructed by \eqref{vorp}, with the associated function $V\colon \R \times [0,+\infty) \to \R^2$.  Henceforth all constants are allowed to depend on $v$.

Let $\delta>0$ be a small quantity to be chosen later, and then let $N_0 > 1$ be a sufficiently large quantity (depending on $\delta$) to be chosen later.  We then form the infinite sequence
$$ N_0 < N_1 < N_2 < \dots$$ 
by the recursive identity
\begin{equation}\label{nini}
 N_i := N_{i-1}^{5/2},
\end{equation}
thus $N_i = N_0^{(5/2)^i}$ for all $i$.  We then select a solution $u: (0,\frac{\delta}{N_0}] \times \R^{11} \to \R$ by the explicit formula
\begin{equation}\label{utai}
 u(t,x) := \sum_{i=1}^\infty u_i(t,x)
\end{equation}
where
$$ u_i(t,x) := N_i^{3/2} v( N_i t, N_i x ) \eta\left( \frac{t N_{i-1}}{\delta} \right) $$
where $\eta\colon \R \to [0,1]$ is a fixed smooth function (not depending on $\delta,N_0$) that is supported on $[-2,2]$ and equals $1$ on $[-1,1]$; cf. Figure \ref{fig:freq}.  The exponents $3/2$ and $5/2$ are the ones predicted by the numerology \eqref{alpha-dig}, \eqref{nj-form} discussed in the introduction.  

Observe that on any compact subset of $(0,\frac{\delta}{N_0}] \times \R^{11}$, only finitely many of the $u_i(t,x)$ are not identically zero, which implies in particular that the sum in \eqref{utai} is absolutely convergent, and that $u$ is smooth on $(0,\frac{\delta}{N_0}] \times \R^{11}$.  By Proposition \ref{vorp}, $v(t,x)$ is supported in a region of the form 
\begin{equation}\label{region}
\{ (t,x): |x| = t + O(1)\}
\end{equation}
where the implied constant in the $O(1)$ notation depends only on $v$.  This implies that $u(t)$ is compactly supported for each $0 < t \leq 1$.

Let $j$ be a natural number.  At the point $(t,x) = (\frac{\delta}{N_j}, 0)$, the fields $u_i(t,x)$ vanish for $i > j+1$, while from Taylor expansion we have
$$ u_i\left(\frac{\delta}{N_j}, 0\right) = N_i^{\frac{3}{2}} \left(v(0,0) + O\left( \delta \frac{N_{i-1}}{N_j} \right) \right)$$
for $i \leq j+1$.  By Proposition \ref{vorp}(iii), we have $|v(0,0)| \gtrsim 1$.  For $\delta$ small enough and $N_0$ large enough, this gives the bound
$$ \left|u\left(\frac{\delta}{N_j}, 0\right)\right| \gtrsim N_{j+1}^{\frac{3}{2}}.$$
Sending $j \to \infty$, we conclude that $u$ cannot be smoothly extended to the spacetime origin $(0,0)$.

Since all of the $u_i$ are spherically symmetric, $u$ is also.  Indeed if $U: (0,\frac{\delta}{N_0}] \times [0,+\infty) \to \R^2$ is the function defined by
\begin{equation}\label{U-def}
 U(t, |x|^2) := u(t,x)
\end{equation}
then we have
$$ U(t,y) = \sum_{i=1}^\infty U_i(t,y)$$
where the functions $U_i: (0, \frac{\delta}{N_0}] \times [0,+\infty) \to \R^2$ are given by the formula
\begin{equation}\label{uit}
 U_i(t,y) := N_i^{\frac{3}{2}} V( N_i t, N_i^2 y ) \eta\left( \frac{t N_{i-1}}{\delta} \right).
\end{equation}
Note from the support of $V, \eta$ that $U_i$ is supported in the region
\begin{equation}\label{u-supp}
\left\{ (t,y): t \lesssim \frac{\delta}{N_{i-1}}; y = (t + O(N_i^{-1}))^2 \right\}
\end{equation}

Our remaining task is to locate a smooth function $f\colon \R^m \to \R^m$ with all derivatives bounded, such that \eqref{boxu} holds on $(0,\frac{\delta}{N_0}] \times \R^{11}$.  Differentiating \eqref{utai} term-by-term (which is justified on any compact subset of $(0,\frac{\delta}{N_0}] \times \R^{11}$ as there are only finitely many non-zero terms), and noting that $u_1$ solves the free wave equation on $(0,\frac{\delta}{N_0}] \times \R^{11}$, we have
$$
\Box u(t,x) = \sum_{i=2}^\infty \Box u_i(t,x).$$
Since $v$ solves the wave equation, we see from the product rule that
$$ \Box u_i(t,x) := -2 N_i^{5/2} \frac{N_{i-1}}{\delta} (\partial_t v)( N_i t, N_i x ) \eta'\left( \frac{t N_{i-1}}{\delta} \right) 
- N_i^{3/2} \frac{N_{i-1}^2}{\delta^2} v(N_i t, N_i x) \eta''\left(\frac{t N_{i-1}}{\delta} \right).$$
Setting $F: (0,1] \times [0,+\infty) \to \R^2$ to be the function
\begin{equation}\label{F-def}
 F( t, |x|^2 ) := \Box u(t,x)
\end{equation}
we thus have
$$ F(t, y) = \sum_{i=2}^\infty F_i(t,y)$$
where
\begin{equation}\label{dam}
 F_i(t,y) := -2 N_i^{5/2} \frac{N_{i-1}}{\delta} (\partial_t V)( N_i t, N_i^2 y ) \eta'\left( \frac{t N_{i-1}}{\delta} \right)
- N_i^{3/2} \frac{N_{i-1}^2}{\delta^2} V(N_i t, N_i^2 y) \eta''\left(\frac{t N_{i-1}}{\delta} \right).
\end{equation}
In practice, the first term on the right-hand side of \eqref{dam} will dominate the second.

From \eqref{dam} and the support \eqref{region} of $v$, we see (for $N_0$ large enough) that $F_i(t,x)$ is only non-vanishing on the rectangle $R_i$ defined as
$$ R_i := \left\{ (t,y): \frac{\delta}{N_{i-1}} \leq t \leq 2 \frac{\delta}{N_{i-1}}; \quad \left(\frac{\delta}{2N_{i-1}}\right)^2 \leq y \leq \left(3 \frac{\delta}{N_{i-1}}\right)^2 \right\}.$$
Observe (if $\delta$ is small enough and $N_0$ large enough) that the $R_i$ are disjoint, thus $F = F_i$ on $R_i$, and $F$ vanishes outside of $\bigcup_i R_i$.  We have the following additional properties:

\begin{proposition}[Behaviour of $U,F$ on $R_i$]\label{urf}  Let $i \geq 2$.
\begin{itemize}
\item[(i)]  The map $U$ is a diffeomorphism from $R_i$ to $U(R_i)$.  In fact, for any constant $C>1$, if we set $R_{i,C}$ to be the rectangle
$$ R_{i,C} := \left\{ (t,y): C^{-1} \frac{\delta}{N_{i-1}} \leq t \leq C \frac{\delta}{N_{i-1}}; 0 \leq y \leq \left(C \frac{\delta}{N_{i-1}}\right)^2 \right\}$$
then (if $\delta$ is sufficiently small depending on $C$, and $N_0$ sufficiently large) $U$ is a diffeomorphism from $R_{i,C}$ to $U(R_{i,C})$.
\item[(ii)]  One has $U^{-1}(U(R_i)) = R_i$.  In other words, if $(t,y) \in (0,1] \times [0,+\infty)$ is such that $U(t,y) \in U(R_i)$, then $(t,y) \in R_i$.
\item[(iii)]  For $(t,y) \in R_i$, we have the derivative bounds
\begin{equation}\label{f-deriv}
 | (N_{i}^{-1} \partial_t)^j (N_i^{-1} N_{i-1}^{-1} \partial_y)^k F(t,y) | \lesssim_{\delta,j,k} 1,
\end{equation}
for all $j,k \geq 0$; if $(j,k) \neq 0$, we also have
\begin{equation}\label{u-deriv}
 | (N_{i}^{-1} \partial_t)^j (N_i^{-1} N_{i-1}^{-1} \partial_y)^k U(t,y) | \lesssim_{\delta,j,k} 1.
\end{equation}
We also have the non-degeneracy condition
\begin{equation}\label{u-deg}
 | N_{i}^{-1} \partial_t U(t,y) \wedge N_i^{-1} N_{i-1}^{-1} \partial_y U(t,y)| \gtrsim_\delta 1.
\end{equation}
where $(a,b) \wedge (c,d) := ad-bc$ is the wedge product on $\R^2$.
\end{itemize}
\end{proposition}

\begin{proof}  We begin with \eqref{u-deriv}.  Let $j,k \geq 0$, with $(j,k) \neq (0,0)$.  Using \eqref{nini}, we see that our task is to show that
\begin{equation}\label{dandy}
 |\partial_t^j \partial_y^k U| \lesssim_{\delta,j,k} N_{i-1}^{\frac{5}{2}j + \frac{7}{2}k}
\end{equation}
on $R_i$.  In fact we will show these bounds on the larger rectangle $R_{i,C}$ (assuming $\delta$ sufficiently small depending on $C$, and $N_0$ sufficiently large, and allowing implied constants to depend on $C$).  From the support \eqref{u-supp}, the fields $U_{i'}$ vanish on $R_{i,C}$ for $i'>i$, so we have
$$
|\partial_t^j \partial_y^k U|  \leq |\partial_t^j \partial_y^k U_i| + \sum_{i'=1}^{i-1} |\partial_t^j \partial_y^k U_{i'}|.$$
Applying the product rule and chain rule to \eqref{uit}, we have
$$ |\partial_t^j \partial_y^k U_i| \lesssim_{j,k,\delta} \sum_{j'=0}^j N_i^{\frac{3}{2} + j' + 2k} N_{i-1}^{j-j'} (\partial_t^{j'} \partial_y^k V)(N_i t, N_i^2 y) \eta^{(j-j')}\left( \frac{t N_{i-1}}{\delta} \right)$$
Applying \eqref{vos} and then \eqref{nini}, we conclude that
\begin{align*}
 |\partial_t^j \partial_y^k U_i| &\lesssim_{j,k,\delta} \sum_{j'=0}^j N_i^{\frac{3}{2} + j' + 2k} N_{i-1}^{j-j'} (N_i/N_{i-1})^{-\frac{d-1}{2}-k} \\
& \lesssim_{j,k,\delta} N_i^{\frac{3}{2} + j + 2k} (N_i/N_{i-1})^{-5-k} \\
&= N_{i-1}^{\frac{5}{2}j + \frac{7}{2}k - \frac{15}{4}}.
\end{align*}
Thus this contribution to \eqref{dandy} is acceptable (with some room to spare).  Thus it remains to show that
$$
 \sum_{i'=1}^{i-1} |\partial_t^j \partial_y^k U_{i'}| \lesssim_{\delta,j,k} N_{i-1}^{\frac{5}{2}j + \frac{7}{2}k} $$
on $R_{i,C}$.  For $1 \leq i' \leq i-1$, the $\eta$ factor in \eqref{uit} for $U_{i'}$ is locally equal to $1$, thus
$$ \partial_t^j \partial_y^k U_{i'}(t,y) =  N_{i'}^{\frac{3}{2}+j+2k} (\partial_t^j \partial_y^k) V( N_{i'} t, N_{i'}^2 y )$$
and thus by \eqref{vos} and \eqref{nini} we have
\begin{align*}
 \sum_{i'=1}^{i-1} |\partial_t^j \partial_y^k U_{i'}| &\lesssim_{\delta,j,k} \sum_{i'=1}^{i-1} N_{i'}^{\frac{3}{2}+j+2k} \\
&\lesssim_{\delta,j,k} N_{i-1}^{\frac{3}{2}+j+2k} \\
&\lesssim_{\delta,j,k} N_{i-1}^{\frac{5}{2}j + \frac{7}{2}k}
\end{align*}
as required, where we have used $(j,k) \neq (0,0)$ in the last line.  This establishes \eqref{u-deriv}.

Applying the above analysis to $(j,k) = (1,0), (0,1)$, we see that the $U_{i-1}$ term dominates, with
$$ \partial_t U(t,y) = N_{i-1}^{\frac{5}{2}} (\partial_t V)( N_{i-1} t, N_{i-1}^2 y ) + O_\delta( N_{i-1}^{\frac{5}{2} - \frac{15}{4} }) + O_\delta( N_{i-2}^{\frac{5}{2}} ) $$
and similarly
$$ \partial_y U(t,y) = N_{i-1}^{\frac{7}{2}} (\partial_y V)( N_{i-1} t, N_{i-1}^2 y ) + O_\delta( N_{i-1}^{\frac{7}{2} - \frac{15}{4} } ) + O_\delta( N_{i-2}^{\frac{7}{2}} ) $$
on $R_{i,C}$ (with the $N_{i-2}$ term deleted if $i=1$).  Since $(N_{i-1}t, N_{i-1}^2 y)$ is within $O(\delta)$ of the origin, we conclude (for $N_0$ large enough) that
\begin{equation}\label{uta}
 \partial_t U(t,y) = N_{i-1}^{\frac{5}{2}} ( (\partial_t V)(0,0) + O(\delta) )
\end{equation}
and
\begin{equation}\label{utb}
 \partial_y U(t,y) = N_{i-1}^{\frac{7}{2}} ( (\partial_y V)(0,0) + O(\delta) )
\end{equation}
on $R_{i,C}$.  In particular we have
$$ N_{i}^{-1} \partial_t U(t,y) \wedge N_i^{-1} N_{i-1}^{-1} \partial_y U(t,y) = (\partial_t V \wedge \partial_y V)(0,0 ) + O(\delta).$$
From Proposition \ref{vorp}(iv), the expression $(\partial_t V \wedge \partial_y V)(0,0 )$ is non-zero, and \eqref{u-deg} follows if $\delta$ is small enough.  

Also, if we integrate \eqref{uta}, \eqref{utb} we see that
\begin{align*}
 U(t,y) - U(t',y') &= N_{i-1}^{\frac{3}{2}} (  (\partial_t V)(0,0) N_{i-1} (t-t') + (\partial_y V)(0,0) N_{i-1}^2 (y-y') \\
&\quad + O( \delta (|N_{i-1} (t-t')| + |N_{i-1}^2 (y-y')| ) ) )
\end{align*}
for $(t,y), (t',y') \in R_{i,C}$.  From Proposition \ref{vorp}(v) we have
$$ |(\partial_t V)(0,0) N_{i-1} (t-t') + (\partial_y V)(0,0) N_{i-1}^2 (y-y')| \gtrsim |N_{i-1} (t-t')| + |N_{i-1}^2 (y-y')|$$
and so we conclude that $U(t,y) \neq U(t',y')$ if $(t,y) \neq (t',y')$, thus $U$ is injective on $R_{i,C}$.  Combining this with \eqref{u-deg} we obtain the conclusion (i).

Next, we establish \eqref{f-deriv}.  On $R_i$, $F$ is equal to $F_i$, and so it suffices to show that
\begin{equation}\label{dandy-2}
 |\partial_t^j \partial_y^k F_i| \lesssim_{\delta,j,k} N_{i-1}^{\frac{5}{2} j + \frac{7}{2}k}
\end{equation}
Applying \eqref{dam} and the product rule, using the fact that all derivatives of $\eta$ are bounded, followed by \eqref{vos} and \eqref{nini}, we have
\begin{align*}
 |\partial_t^j \partial_y^k F_i| &\lesssim_{\delta,j,k} \sum_{j'=0}^j N_i^{\frac{5}{2}} N_{i-1} N_i^{j'+2k} N_{i-1}^{j-j'} |(\partial_t^{j'+1} \partial_y^k V)(N_i t, N_i^2 y)| \\
&\quad + N_i^{\frac{3}{2}} N_{i-1}^2  N_i^{j'+2k} N_{i-1}^{j-j'} |(\partial_t^{j'} \partial_y^k V)(N_i t, N_i^2 y) | \\
&\lesssim_{\delta,j,k} \sum_{j'=0}^j N_i^{\frac{5}{2}} N_{i-1} N_i^{j'+2k} N_{i-1}^{j-j'} (N_i/N_{i-1})^{-5-k} \\
&\quad + N_i^{\frac{3}{2}} N_{i-1}^2  N_i^{j'+2k} N_{i-1}^{j-j'} (N_i/N_{i-1})^{-5-k} \\
&\lesssim_{\delta,j,k} N_i^{\frac{5}{2}} N_{i-1} N_i^{j+2k} (N_i/N_{i-1})^{-5-k}  + N_i^{\frac{3}{2}} N_{i-1}^2  N_i^{j+2k} (N_i/N_{i-1})^{-5-k} \\
&\lesssim_{\delta,j,k} N_i^{\frac{5}{2}} N_{i-1} N_i^{j+2k} (N_i/N_{i-1})^{-5-k} \\
&\lesssim_{\delta,j,k} N_{i-1}^{\frac{5}{2} j + \frac{7}{2} j - \frac{1}{4}}.
\end{align*}
Crucially, the term $\frac{1}{4}$ is non-negative (this is the $d=11$ case of \eqref{d10}), and the claim \eqref{f-deriv} follows.

Finally, we establish (ii), which is the most difficult claim to establish.  Suppose for contradiction that we had $(t,y) \in R_i$ and $(t',y') \not \in R_i$ such that $U(t,y) = U(t',y')$.  By the already established bounds on $U_{i'}(t,y)$, we see that
\begin{align}
 U(t,y) &= U_{i-1}(t,y) + O_{\delta}( N_{i-1}^{\frac{3}{2} - \frac{15}{4}} ) + O_\delta( N_{i-2}^{\frac{3}{2}} ) \\
&=U_{i-1}(t,y) + O(\delta^3 N_{i-1}^{\frac{3}{2}})\label{mats}
\end{align}
(say); by \eqref{uit} and the fact that $N_{i-1} t = O(\delta)$ and $N_{i-1}^2 y = O(\delta^2)$ is within $O(\delta)$ of the origin, we conclude that
$$
U(t,y) = N_{i-1}^{\frac{3}{2}} (V( N_{i-1} t, N_{i-1}^2 y ) + O(\delta^3)).
$$
Since $U(t',y') = U(t,y)$, we thus have
\begin{equation}\label{drib}
 U(t',y') = N_{i-1}^{\frac{3}{2}} (V(N_{i-1} t, N_{i-1}^2 y ) + O(\delta^3)) \sim N_{i-1}^{\frac{3}{2}}.
\end{equation}
Let $J$ be the largest natural number for which $U_J(t',y')$ is non-zero; this quantity is finite from \eqref{u-supp} since $t'>0$, and positive since $U(t',y')$ is non-zero.  Then from \eqref{u-supp} we have
\begin{equation}\label{sat}
t' \lesssim \frac{\delta}{N_{J-1}}
\end{equation}
and
\begin{equation}\label{say}
 y' \lesssim (t')^2 + \frac{1}{N_J^2}.
\end{equation}
In particular, $(N_j t', N_j^2 y') = O(\delta)$ and $\frac{t' N_{j-1}}{\delta} < 1$ for $j < J$, if $N_0$ is large enough; from this and \eqref{uit} we have
$$ U_j(t',y') = N_j^{\frac{3}{2}} (V(0,0) + O(\delta))$$
for $j < J$.  From the first half of Proposition \ref{vorp}(iii) we conclude (for $\delta$ small enough) that both components of $U_j(t',y')$ are positive and $\gtrsim N_j^{\frac{3}{2}}$ for all $j < J$.  On the other hand, from the second half of Proposition \ref{vorp}(iii) we know that at least one of the components of $U_J(t',y')$ is non-negative.  Summing, we conclude that at least one of the components of $\sum_{j=1}^J U_j(t',y')$ is positive and $\gtrsim N_{J-1}^{\frac{3}{2}}$.  Comparing this with \eqref{drib}, we see that $J \leq i$.  
In particular
$$ U(t',y') = \sum_{j=1}^i U_j(t',y').$$
On the other hand, from \eqref{uit}, \eqref{vos} one has
$$ U_j(t',y') \lesssim N_j^{\frac{3}{2}}$$
for $j < i-1$, and thus
\begin{align}
U(t',y') &= U_{i-1}(t',y') + U_i(t',y') + O( N_{i-2}^{\frac{3}{2}} )\nonumber \\
&= U_{i-1}(t',y') + U_i(t',y') + O( \delta^3 N_{i-1}^{\frac{3}{2}} ) \label{soso}
\end{align}
(say), if $N_0$ is large enough.

We in fact claim that
\begin{equation}\label{season}
U(t',y') = U_{i-1}(t',y') + O( \delta^3 N_{i-1}^{\frac{3}{2}} ).
\end{equation}
If $J<i$, then $U_i(t',y')=0$ and the claim \eqref{season} is immediate from \eqref{soso}.  Now suppose instead that $J=i$.
From Proposition \ref{vorp}(iii) we know that at least one of the components $U_{i,1}(t',y'), U_{i,2}(t',y')$ of $U_i(t',y')$ is non-negative, thus we have
$$ U_{,a}(t',y') \geq U_{i-1,a}(t',y') + O( \delta^3 N_{i-1}^{\frac{3}{2}} )$$
for some $a=1,2$, where $U_{,1}, U_{,2}$ are the components of $U$.
On the other hand, from \eqref{sat}, \eqref{say} we have $N_{i-1} t' \lesssim \delta$ and $N_{i-1}^2 y' \lesssim \delta^2$, so from \eqref{uit} we have
$$ U_{i-1}(t',y') = N_{i-1}^{\frac{3}{2}} ( V(N_{i-1} t', N_{i-1}^2 y') + O(\delta^3) ).$$
Comparing this with \eqref{drib}, we conclude that
$$ V_a( N_{i-1} t, N_{i-1}^2 y ) \geq V_a( N_{i-1} t', N_{i-1}^2 y' ) + O(\delta^3).$$
Using \eqref{vat}, \eqref{vat-2}, we conclude that either
$$ N_{i-1} t' + N_{i-1}^2 y'  \gtrsim N_{i-1} t +  N_{i-1}^2 y - O(\delta^3 )$$
or
$$ (N_{i-1} t')^2 + N_{i-1}^2 y'  \gtrsim (N_{i-1} t)^2 +  N_{i-1}^2 y - O(\delta^3 ).$$
In either case we have either $N_{i-1} t' \gtrsim \delta$ or $N_{i-1}^2 y' \gtrsim \delta^2$; actually, from \eqref{say} the latter estimate implies the former, thus $N_{i-1} t' \gtrsim \delta$.  We now see from \eqref{uit}, \eqref{vos} that
\begin{align*}
 |U_i(t',y')| &\lesssim N_i^{\frac{3}{2}} (\delta N_i / N_{i-1})^{-5} \\
&\lesssim \delta^3 N_{i-1}^{\frac{3}{2}},
\end{align*}
and \eqref{season} follows.

From \eqref{season}, \eqref{drib} we have
\begin{equation}\label{sss}
 U_{i-1}(t',y') = N_{i-1}^{\frac{3}{2}} (V(N_{i-1} t, N_{i-1}^2 y ) + O(\delta^3)) \sim N_{i-1}^{\frac{3}{2}}.
\end{equation}
From \eqref{uit}, \eqref{vos} this implies that
$$ t' = O( N_{i-1}^{-1} ) $$
and (from the support of $V$)
$$ y' = O( N_{i-1}^{-2} ).$$
In particular, from \eqref{uit} we have
$$ U_{i-1}(t',y') = N_{i-1}^{\frac{3}{2}} V( N_{i-1} t', N_{i-1}^2 y' ).$$
From \eqref{sss} we conclude that
$$ V( N_{i-1} t', N_{i-1}^2 y' ) = V( N_{i-1} t, N_{i-1}^2 y ) + O( \delta^3 ).$$
Applying \eqref{vat}, \eqref{vat-2} we conclude that
$$ N_{i-1} t' + N_{i-1}^2 y' \sim N_{i-1} t + N_{i-1}^2 y + O( \delta^3 ) \sim \delta$$
and
$$ (N_{i-1} t')^2 + N_{i-1}^2 y' \sim (N_{i-1} t)^2 + N_{i-1}^2 y + O( \delta^3 ) \sim \delta^2$$
and thus
$$ t' \sim \delta N_{i-1}^{-1} $$
and 
$$ y' = O( \delta^2 N_{i-1}^{-2} ).$$
Thus $(t,y), (t',y')$ both lie in $R_{i,C}$ for some fixed $C$.  But by part (i), this forces $(t,y) = (t',y')$, a contradiction.  This concludes the proof of (ii).
\end{proof}

We are now ready to conclude the proof of Theorem \ref{main}.  We define the function $f\colon \R^2 \to \R^2$ by setting
\begin{equation}\label{fui}
 f( U( t, y ) ) := F(t,y)
\end{equation}
whenever $(t,y) \in R_i$ for some $i$, with $f$ vanishing outside of $\bigcup_{i=1}^\infty U(R_i)$.  From parts (i) and (ii) of Proposition \ref{urf} we see that $f$ is well-defined and smooth, and from \eqref{U-def}, \eqref{F-def} we see that $u$ and $f$ obey the equation \eqref{boxu}.  The only remaining task is to show that $f$ and all of its derivatives are bounded.  By \eqref{fui} it suffices to show that the composition map $F \circ U^{-1} F: U(R_i) \to \R^2$ has all derivatives bounded uniformly in $i$, using Proposition \ref{urf}(i) to construct an inverse map $U^{-1} \colon U(R_i) \to R_i$. (Note that we can allow the bound to depend on $\delta$ and on the number of derivatives used.)

It is convenient to work on the renormalised rectangle
$$ \tilde R_i := \{ (N_i t, N_i N_{i-1} y): (t,y) \in R_i \}$$
with the renormalised functions $\tilde U, \tilde F \colon \tilde R_i \to \R^2$ given by
$$ \tilde U(\tilde t,\tilde y) := U\left( \frac{\tilde t}{N_i}, \frac{\tilde y}{N_i N_{i-1}} \right) $$
and
$$ \tilde F(\tilde t,\tilde y) \coloneqq F\left( \frac{\tilde t}{N_i}, \frac{\tilde y}{N_i N_{i-1}} \right) $$
for $(\tilde t, \tilde y) \in \tilde R_i$.  Clearly $U(R_i) = \tilde U( \tilde R_i)$ and $F \circ U^{-1}=  \tilde F \circ \tilde U^{-1}$, so it suffices to show the pointwise estimates
\begin{equation}\label{r2j}
 |\nabla_{\R^2}^j ( \tilde F \circ \tilde U^{-1} )| \lesssim_{\delta,j} 1 
\end{equation}
on $\tilde U(\tilde R_i)$ for all $j \geq 0$.  From \eqref{f-deriv} and the chain rule we have
\begin{equation}\label{fu}
 | \nabla_{\tilde t,\tilde y}^j \tilde F(\tilde t,\tilde y) | \lesssim_{\delta,j} 1,
\end{equation}
on $\tilde R_i$ for all $j \geq 0$; if $j \geq 1$, we also have from \eqref{u-deriv} that
\begin{equation}\label{tu}
 | \nabla_{\tilde t,\tilde y}^j \tilde U(\tilde t,\tilde y) | \lesssim_{\delta,j} 1
\end{equation}
on $\tilde R_i$.  Finally, from \eqref{u-deg} we have
\begin{equation}\label{tu-2}
 | \partial_{\tilde t} \tilde U(\tilde t,\tilde y) \wedge \partial_{\tilde y} \tilde U(\tilde t,\tilde y)| \gtrsim_\delta 1.
\end{equation}
From the inverse function theorem and many applications of the chain and product rules, we see from \eqref{tu}, \eqref{tu-2} that
$$ |\nabla_{\R^2}^j \tilde U^{-1} | \lesssim_{\delta,j} 1$$
on $\tilde U(\tilde R_i)$ for all $j \geq 1$; combining this with \eqref{fu} and many more applications of the chain and product rules, we obtain \eqref{r2j} as desired.  This completes the proof of Theorem \ref{main}.

\appendix

\section{Global regularity in low dimension}\label{reg}

In this appendix we show global regularity for the equation \eqref{boxu} in dimensions $d \leq 9$.  Actually, to simplify the exposition we shall only handle the most difficult case $d=9$; dimensions lower than $9$ can be handled by a suitable modification of the numerology below, or else by adding some dummy spatial variables to extend the solution $u$ (and the initial data $u_0, u_1$), and then using finite speed of propagation to smoothly truncate the solution to again be compactly supported in space; we leave the details to the interested reader.

Fix $m, f$.  Assume for sake of contradiction that global regularity failed, then there exists $0 < T_* < \infty$ and a smooth solution $u \colon [0,T_*) \times \R^9 \to \R^m$ to \eqref{boxu} which is compactly supported in space, but which does not lie in $L^\infty_t L^\infty_x( [0,T_*) \times \R^9)$.

For each $s \geq 0$, let $P(s)$ denote the claim that the nonlinearity $f(u)$ lies in $L^1_t H^s_x( [0,T_*) \times \R^9)$.  From the discussion in the introduction, we know that $P(1)$ holds.  We will show the following recursion:

\begin{theorem}  For any $1 \leq s \leq \frac{7}{2}$, the claim $P(s)$ implies $P(s')$ for some $s' = s'(s) > s$ that depends continuously on $s$.
\end{theorem}

Iterating this claim (noting that $s'(s)-s$ will be bounded from below on $[1,\frac{7}{2}]$), we conclude that $P(s)$ holds for some $s > \frac{7}{2}$, which by \eqref{boxu} and energy estimates implies that $u$ lies in $L^\infty_t H^s_x([0,T_*) \times \R^9)$ for some $s > \frac{9}{2}$, and the desired contradiction then follows from Sobolev embedding.

It remains to prove the theorem. To simplify the notation we omit the domains $\R^9$ and $[0,T_*) \times \R^9$ from the Lebesgue norms in the estimates below.
 Fix $1 \leq s \leq \frac{7}{2}$.  Let $2 \leq p < \infty$ be an exponent depending continuously on $s$ to be chosen later.  Let $\eps>0$ be a small quantity depending continuously on $s,p$ to be chosen later, and then $J$ be a large integer depending on $s,p,\eps$ to be chosen later.  For any $N \geq 1$, we let $P_N$ be a smooth Fourier projection of Littlewood-Paley\footnote{See e.g. \cite[Appendix A]{tao-book} for the basic theory of Littlewood-Paley projections that are needed here.} type to $\{ \xi: N \leq 1 + |\xi| \leq 4N \}$, in such a fashion that $\sum_N P_N$ is the identity when $N$ ranges over powers of two. We now use $X \lesssim Y$, $Y \gtrsim X$, or $X = O(Y)$ to denote a bound of the form $|X| \leq CY$ where $C$ can depend on $s, p, \eps, J, m, f, T_*, u_0, u_1, u$ but does not depend on additional variables such as $N$.  Applying $P_N$ to the hypothesis $P(s)$, we conclude in particular that
$$ \| P_N( f(u) ) \|_{L^1_t L^2_x} \lesssim N^{-s} $$
for all $N$.  Also, since the initial data $u_0, u_1$ is smooth, one has
$$ N \| P_N u_0 \|_{L^2_x} + \| P_N u_1 \|_{L^2_x} \lesssim N^{-s}.$$
On the other hand, applying $P_N$ to \eqref{boxu} we see that
$$ \Box (P_N u) = P_N( f(u) ) $$
and that $P_N u$ has initial position $P_N u_0$ and initial velocity $P_N u_1$.  Energy estimates then give
\begin{equation}\label{energy} 
\| P_N u \|_{L^\infty_t L^2_x} \lesssim N^{-s - 1}
\end{equation}
for any $N \geq 1$.
Meanwhile, the endpoint\footnote{One could also use the simpler non-endpoint Strichartz estimates (see e.g. \cite{sogge-wave}) here if desired, after adjusting the exponents by an epsilon, to obtain the same conclusions of global regularity for \eqref{boxu}; we leave the details to the interested reader.} Strichartz estimate \cite[Corollary 1.3]{tao-keel} (with $q=2, r = 8/3, \gamma = 5/8$) asserts that
$$ \| u \|_{L^2_t L^{8/3}_x} \lesssim \| f \|_{\dot H^{5/8}_x} + \|g\|_{\dot H^{-3/8}_x}$$
whenever $\Box u = 0$ with $u(0)=f$ and $u_t(0)=g$, which by Duhamel's formula and Minkowski's inequality gives
$$ \| u \|_{L^2_t L^{8/3}_x} \lesssim \| f \|_{\dot H^{5/8}_x} + \|g\|_{\dot H^{-3/8}_x} + \| F \|_{L^1_t \dot H^{-3/8}_x}$$
whenever $\Box u = F$ with $u(0)=f$ and $u_t(0)=g$; applying this with $u,f,g,F$ replaced by $P_N u, P_N u_0, P_N u_1, P_N(f(u))$ gives the estimate
$$ \| P_Nu \|_{L^2_t L^{8/3}_x} \lesssim N^{-s-3/8},$$
for any $N \geq 1$, which from Bernstein's inequality yields
\begin{equation}\label{strick}
 \|P_N u \|_{L^2_t L^4_x} \lesssim N^{-s+3/4}.
\end{equation}
Meanwhile, from \eqref{energy} and the Bernstein and Holder inequalities we have
\begin{equation}\label{strick-inf} 
\| P_N u \|_{L^\infty_t L^{\infty}_x} \lesssim N^{-s + \frac{7}{2}}
\end{equation}
and
\begin{equation}\label{strick-en} 
\| P_N u \|_{L^1_t L^2_x} \lesssim N^{-s - 1}.  
\end{equation}
By interpolation between \eqref{strick}, \eqref{strick-inf}, and \eqref{strick-en}, we have
\begin{equation}\label{cap}
\| P_N u \|_{L^{p/2}_t L^p_x} \lesssim N^{-s - c(p)}
\end{equation}
for all $2 \leq p \leq \infty$, where
where
\begin{equation}\label{cp-1}
 c(p)\coloneqq \min\left( \frac{7}{p} - \frac{5}{2}, \frac{11}{p}-\frac{7}{2} \right)
\end{equation}
In particular, upon dyadic summation we see that
\begin{equation}\label{cap-2}
\| \nabla u \|_{L^{p/2}_t L^p_x} \lesssim 1
\end{equation}
whenever $2 \leq p \leq \infty$ is such that $c(p) > 1-s$.  The utility of this estimate will become clearer shortly.

To prove $P(s')$ for some $s'>s$ depending continuously on $s$, it suffices to establish the bound
\begin{equation}\label{pin}
\| P_N(f(u)) \|_{L^1_t L^2_x} \lesssim N^{-s-\eps}
\end{equation}
for all $N \geq 2$ (since the $N=1$ contribution can be handled by the existing hypothesis $P(s)$).  Accordingly, let us fix $N \geq 2$.
We split
$$ u = u_{<N^{1-\eps}} + u_{\geq N^{1-\eps}}$$
where 
$$ u_{<N^{1-\eps}} = \sum_{M < N^{1-\eps}} P_M u$$
and
$$ u_{\geq N^{1-\eps}} = \sum_{M \ge N^{1-\eps}} P_M u$$
where $M$ ranges over powers of two.  By the triangle inequality we have
$$
\| P_N(f(u)) \|_{L^1_t L^2_x} \leq
\| P_N(f(u_{<N^{1-\eps}})) \|_{L^1_t L^2_x} + \| P_N(f(u) - f(u_{<N^{1-\eps}}) \|_{L^1_t L^2_x}.$$

From Plancherel's theorem and the Lipschitz bound $|f(u) - f(u_{<N^{1-\eps}})| \lesssim u_{\geq N^{1-\eps}}$, followed by \eqref{strick-inf}, one has
\begin{align*}
\| P_N(f(u) - f(u_{<N^{1-\eps}}) \|_{L^1_t L^2_x} \\
&\lesssim \| f(u) - f(u_{<N^{1-\eps}}) \|_{L^1_t L^2_x} \\
&\lesssim \| u_{\geq N^{1-\eps}} \|_{L^1_t L^2_x} \\
&\lesssim N^{-(s+1)(1-\eps)}\\
&\lesssim N^{-s-\eps}
\end{align*}
if $\eps$ is small enough.
Thus it will now suffice to show
\begin{equation}\label{pin-2}
\| P_N(f(u_{<N^{1-\eps}})) \|_{L^1_t L^2_x} \lesssim N^{-s-\eps}.
\end{equation}
By Plancherel's theorem we have
$$ 
\| P_N(f(u_{<N^{1-\eps}})) \|_{L^1_t L^2_x} \lesssim N^{-J} \| P_N(\nabla^J f(u_{<N^{1-\eps}})) \|_{L^1_t L^2_x}.$$
Since $f(u_{<N^{1-\eps}}) = O(1)$, we also have
$$ N^{-J} |P_N(\nabla^J f(u_{<N^{1-\eps}})) \lesssim 1 $$
and thus
$$ 
\| P_N(f(u_{<N^{1-\eps}})) \|_{L^1_t L^2_x} \lesssim \| \min( 1, N^{-J} |P_N(\nabla^J f(u_{<N^{1-\eps}}))| ) \|_{L^1_t L^2_x}.$$
By repeated application of the chain rule, and the hypothesis that all derivatives of $f$ are bounded, we can expand $\nabla^J f(u_{<N^{1-\eps}})$ as a sum of $O(1)$ terms $Q$, each of which obeys a bound of the form
\begin{equation}\label{so}
 Q = O( |\nabla^{j_1} u_{<N^{1-\eps}}| \dots |\nabla^{j_k} u_{N^{1-\eps}}|
\end{equation}
for some $k \geq 1$ and $j_1,\dots,j_k \geq 1$ with 
\begin{equation}\label{stab}
j_1+\dots+j_k = J.
\end{equation}
It thus suffices to show that
\begin{equation}\label{pin-3}
\| \min(1, N^{-J} |P_N(Q)|) \|_{L^1_t L^2_x} \lesssim N^{-s-\eps}.
\end{equation}
for each one of these terms $Q$.

There are two cases, depending on whether $k$ is large or small.  First suppose that $k$ is small in the sense that $k \leq p/2$.  In this case, we use \eqref{strick-inf}, \eqref{strick-en} to very crudely bound
$$ \| \nabla^{j_i} u_{\leq N^{1-\eps}} \|_{L^\infty_t L^{\infty}_x},
\| \nabla^{j_i} u_{\leq N^{1-\eps}} \|_{L^1_t L^{2}_x} \lesssim N^{(1-\eps) (j_i + \frac{5}{2})} $$
for all $i=1,\dots,k$, and hence by \eqref{so} and H\"older's inequality
$$ \| Q \|_{L^1_t L^2_x} \lesssim N^{(1-\eps) (J + \frac{5}{2} k)};$$
applying $P_N$ and then $\min(1,\cdot)$ we conclude that
$$ \| \min( 1, N^{-J} |P_N(Q)| ) \|_{L^1_t L^2_x} \lesssim N^{\frac{5}{2} k -\eps J},$$
which gives \eqref{pin-3} for $J$ large enough depending on $\eps,p$, since $k$ is bounded by $p/2$.

Now suppose that $k > p/2$.  Then we may bound
$$ \min(1, N^{-J} |P_N(Q)|) \leq (N^{-J} |P_N(Q)|)^{\frac{p}{2k}}$$
and so to prove \eqref{pin-3} it will suffice to show that
$$
\| (N^{-J} |P_N(Q)|)^{\frac{p}{2k}} \|_{L^1_t L^2_x} \lesssim N^{-s-\eps}$$
which we may rearrange as
\begin{equation}\label{stamina}
\| N^{-J} |P_N(Q)| \|_{L^{p/2k}_t L^{p/k}_x} \lesssim N^{-\frac{2k}{p} (s+\eps)}.
\end{equation}
Using the convolution kernel of $P_N$, we have a pointwise bound
$$ |P_N(Q)| \lesssim |Q| * K_N$$
where $K_N$ is the kernel
$$ K_N(x) := N^9 (1 + N|x|)^{-100k}.$$
By \eqref{so} and H\"older's inequality we have
$$ |Q| * K_N \lesssim \prod_{i=1}^k (|\nabla^{j_i} u_{\leq N^{1-\eps}}|^k * K_N)^{1/k}$$
and by further application of H\"older's inequality, we thus have
$$
\| N^{-J} |P_N(Q)| \|_{L^{p/2k}_t L^{p/k}_x} 
\lesssim N^{-J} \prod_{i=1}^k \| (|\nabla^{j_i} u_{\leq N^{1-\eps}}|^k * K_N)^{1/k} \|_{L^{p/2}_t L^{p}_x}.$$
Using the frequency support of $\nabla^{j_i} u_{\leq N^{1-\eps}}$, we have a pointwise bound
$$ |\nabla^{j_i} u_{\leq N^{1-\eps}}| \lesssim |\nabla^{j_i} u_{\leq N^{1-\eps}}| * K_N.$$
For any function $f\colon \R^9 \to \R$, we have the pointwise bound
$$ |f|*K_N(y) \lesssim (1+N|x-y|)^{10} (|f|*\tilde K_N)(x)$$
for any $x,y$, where
$$ \tilde K_N(x) := N^9 (1+N|x|)^{-10},$$
and hence
$$ ((|f|*K_N)^k * K_N)^{1/k} \lesssim |f|*\tilde K_N.$$
Applying this with $f := |\nabla^{j_i} u_{\leq N^{1-\eps}}|$, we conclude that
$$
\| N^{-J} |P_N(Q)| \|_{L^{p/2k}_t L^{p/k}_x} 
\lesssim N^{-J} \prod_{i=1}^k \| |\nabla^{j_i} u_{\leq N^{1-\eps}}| * \tilde K_N \|_{L^{p/2}_t L^{p}_x}.$$
By Young's inequality we may remove the convolution with $\tilde K_N$.  On frequencies less than $N^{1-\eps}$, the gradient operator $\nabla$ has an $L^p_x$ operator norm of $O(N^{1-\eps})$, so we conclude that
$$
\| N^{-J} |P_N(Q)| \|_{L^{p/2k}_t L^{p/k}_x} 
\lesssim N^{-J} (\prod_{i=1}^k N^{(1-\eps)(j_i-1)}) \| \nabla u_{\leq N^{1-\eps}} \|_{L^{p/2}_t L^{p}_x}^k.$$
If $c(p) > 1-s$, we conclude from \eqref{cap}, \eqref{stab} that
$$
\| N^{-J} |P_N(Q)| \|_{L^{p/2k}_t L^{p/k}_x} 
\lesssim N^{-J} N^{(1-\eps)(J-k)} \lesssim N^{-k}.$$
We thus obtain \eqref{stamina} as long as $p > s/2$, and $\eps$ is small enough.

In summary, we have concluded the desired bound $P(s')$ as long as we can find $2 \leq p < \infty$ such that $c(p) > 1-s$ and $p > 2s$.  By the continuity of $c$, this condition is equivalent to the requirement that
$$ c(2s) > 1-s.$$
But this can be verified from \eqref{cp-1} for all $1 \leq s \leq 7/2$ by a routine calculation (see also Figure \ref{9d}).

\begin{figure} [t]
\centering
\includegraphics{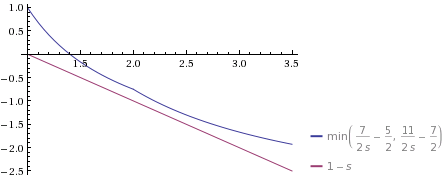}
\caption{A plot of $c(2s)$ and $1-s$ in the range $1 \leq s \leq 7/2$ in the nine-dimensional case.}
\label{9d}
\end{figure}

\begin{remark}  If one were working in $10$ dimensions instead of $9$, then the relevant range of $s$ is now $1 \leq s \leq 4$, and the analogue of the function $c$ is now given by
$$ c(p) := \min( \frac{8}{2p} - 3, \frac{12}{p}-4).$$
The requirement $c(2s) > 1-s$ then fails in the range $2 \leq s \leq 3$, creating a gap in the iterative argument; see Figure \ref{10d}.   To close this gap in ten dimensions, it appears that one needs to go beyond the classical Strichartz estimates; for instance, the Strichartz estimates in amalgam spaces \cite{tao-lowreg} may be of use, although it is not clear to the author if they are able to bridge the gap completely.
\end{remark}

\begin{figure} [t]
\centering
\includegraphics{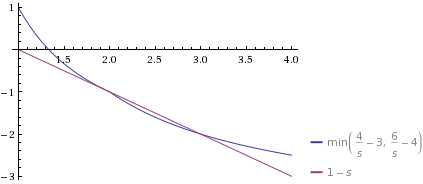}
\caption{A plot of $c(2s)$ and $1-s$ in the range $1 \leq s \leq 4$ in the ten-dimensional case.}
\label{10d}
\end{figure}

\end{document}